\documentclass{amsart}
\usepackage[utf8]{inputenc}
\usepackage[T1]{fontenc}
\title{Linear actions of $\Z/p\times\Z/p$ on $S^{2n-1}\times S^{2n-1}$}
\author{Jim Fowler \and Courtney Thatcher}

\usepackage{amsmath}
\usepackage{amsthm}
\usepackage{amssymb}
\newcommand{\R}{\mathbb R}
\newcommand{\RR}{\mathbb R}
\newcommand{\Q}{\mathbb Q}
\newcommand{\C}{\mathbb C}

\newcommand{\Z}{\mathbb{Z}}
\newcommand{\ZZ}{\Z}
\newcommand{\Zmod}[1]{\mathbb Z_{/#1}}
\newcommand{\Zp}{\Zmod{p}}
\DeclareMathOperator{\Wh}{Wh}
\DeclareMathOperator{\Cat}{Cat}
\DeclareMathOperator{\Top}{Top}
\DeclareMathOperator{\PL}{PL}

\DeclareMathOperator{\GL}{GL}

\newcommand{\GModCat}{G/\hspace{-1.5pt}\Cat}
\newcommand{\GModTop}{G/\hspace{-1.5pt}\Top}
\newcommand{\GModPL}{G/\hspace{-1.5pt}\PL}
\usepackage{todonotes}
\usepackage[normalem]{ulem}

\newcommand{\defnword}[1]{\textbf{#1}}

\newtheorem{theorem}{Theorem}[section]
\newtheorem{corollary}[theorem]{Corollary}

\newtheorem{proposition}[theorem]{Proposition}

\newtheorem{lemma}[theorem]{Lemma}

\newtheorem{remark}[theorem]{Remark}

\usepackage{color}

\usepackage{tikz}
\usepackage{tikz-cd}
\usetikzlibrary{matrix}
\usetikzlibrary{babel}
\usetikzlibrary{cd}

\usepackage{graphicx}
\usepackage[all]{xy}
\xyoption{arc}

\usepackage[style=numeric-comp,doi=false,isbn=false,url=false,eprint=false]{biblatex}
\bibliography{references}

\begin{document}
\begin{abstract}
For an odd prime $p$, we consider free actions of $(\Zp)^2$ on $S^{2n-1}\times S^{2n-1}$ given by linear actions of $(\Zp)^2$ on $\R^{4n}$.  Simple examples include a lens space cross a lens space, but  $k$-invariant calculations show that other quotients exist.  Using the tools of Postnikov towers and surgery theory, the quotients are classified up to homotopy by the $k$-invariants and up to homeomorphism by the Pontrjagin classes.   We will present these results and demonstrate how to calculate the $k$-invariants and the Pontrjagin classes from the rotation numbers.
\end{abstract}
\maketitle

\section{Introduction}

There is a free, linear $\Zp$ action on $S^{2n-1}\subset\C^{n}$ given by $\gamma (x_1,\dots ,x_j) = (e^{2\pi i r_1/p}x_1,\dots ,e^{2\pi i r_j/p}x_j)$. The resulting quotient space is a called a  lens space $L =L(p;r_1,\dots ,r_n)$ and the $r_i$ are called rotation numbers. These spaces are well-known and well-studied.  More generally, one can define a free, linear action of $\Zp$ on a product of \textit{two} spheres, $S^{2n-1}\times S^{2m-1}$, by assigning a set of rotation numbers for each sphere separately. The classification of the resulting quotient spaces can be found in \cite{MR2721633}.

Work has also been done on which groups besides $\Zp$ can act freely on a product of spheres, but less is known about the classification of the quotient spaces associated to those groups known to act freely. In \cite{firstpaper}, free actions of $(\Zp)^2$ on $S^{2n-1}\times S^{2n-1}$ are classified up to homotopy equivalence.  In this paper, we consider free, linear, spherewise actions of $(\Zp)^2$ on $S^{2n-1}\times S^{2n-1}$ with $n > 1$ and with $p$ an odd prime, and classify the resulting quotient spaces up to piecewise linear or topological homeomorphism.  To be more precise about the ``spherewise linear'' actions under consideration, our classification follows a definition from Ray \cite{MR1161304} in which $(\Zp)^2$ acts linearly on each sphere separately with the action being free on at least one of the spheres. One example of such an action is that of a lens space cross a lens space, $L(p;r_{1},\dots ,r_{n})\times L(p;r'_{1},\dots ,r'_{n})$, where the first $\Zp$ acts on the first sphere and the second $\Zp$ factor acts on the second sphere.

There are, however, many more possibilities than lens space cross lens space. In Section~\ref{section:constructions}, we build those more general examples from the representation theory of $(\Zp)^2$.  Then in Section~\ref{section:homotopy} we recall the classification up to homotopy from~\cite{firstpaper} but recast for the specific case of linear spherewise actions, followed by the homeomorphism classification in Section~\ref{section:homeomorphism}.  The paper concludes with an application of the classification theorem presented in Section~\ref{section:application}.

\subsection*{Acknowledgements}

The authors thank the referee and the editors for many comments which significantly improved the paper.

\section{Constructions}
\label{section:constructions}

Throughout this paper, we assume $n > 1$.
The quotient space resulting from the free, linear $\Zp$ action on $S^{2n-1}\subset\C^{n}$ given by $\gamma (x_1,\dots ,x_j) = (e^{2\pi ir_1/p}x_1,\dots ,e^{2\pi ir_j/p}x_j)$ is a called a generalized lens space, often denoted by $L=L(p;r_1,\dots ,r_n)$.  We often conflate a group action with its quotient. The $r_i$ are called rotation numbers, and the homotopy, simple homotopy, and homeomorphism types of generalized lens spaces are determined by various combinations of products of rotation numbers.  We now generalize this type of action to the case of $(\Zp)^2$ acting on $S^{2n-1}\times S^{2n-1}$.

\subsection{Notation}
\label{section:constructon}

Let $\Gamma$ be the group $(\Zp)^2$, where $p \neq 2$ is a
prime.  The goal is to consider free actions of $\Gamma$ on
$S^{2n-1} \times S^{2n-1}$. 

Let $R = (r_1,\ldots,r_n,r'_1,\ldots,r'_n)$ and
$Q = (q_1,\ldots,q_n,q'_1,\ldots,q'_n)$ be elements of $(\Zp)^{2n}$ so
that the span of $R$ and $Q$ yield a $(\Zp)^2$ inside $(\Zp)^{2n}$. The parameters $r_1, \ldots r_n$ and $r'_1,\ldots,r'_n$ and $q_1,\ldots,q_n$ and $q'_1,\ldots,q'_n$ are called ``rotation
numbers'' in analogy with the case of a lens space.  To get an action on a product of spheres, note that the unit $(2n-1)$-sphere in $\R^{2n} \cong \C^n$ gives rise to a product of spheres $S^{2n-1} \times S^{2n-1}$ in the product $\C^n \times \C^n$, and then $R$ acts on $\C^n \times \C^n$ preserving these spheres via
\begin{align*}
R \cdot (z,z') &= (r,r') \cdot (z,z') \\
&= (r,r') \cdot (z_1,\ldots,z_n,z'_1,\ldots,z'_n) \\
&= \left( e^{2\pi i r_1/p} z_1, \ldots, e^{2\pi i r_1/p} z_n, e^{2\pi i r'_1/p} z'_1, \ldots e^{2\pi i r'_n/p} z'_n \right),
\end{align*}
and similarly for $Q$.  As a result, the group 
$(\Zp)^{2} \cong \langle R, Q \rangle$ acts on $S^{2n-1} \times S^{2n-1}$.  This action may not be free, but when the resulting action \textit{is} free, we call these the
\defnword{standard linear examples.}  Classifying these examples is the goal of this paper; we will not consider other possibilities for $(\Zp)^2$ actions on $S^{2n-1} \times S^{2n-1}$.

The tiniest bit of representation theory reveals that this
construction involving $4n$ rotation numbers captures every free
linear action of $(\Zp)^2$ on $S^{2n-1} \times S^{2n-1}$. 

\begin{proposition}
  \label{proposition:representation-theory}
  A representation of $(\Zp)^2$ on $\RR^{2n} \times \RR^{2n}$
  preserving the decomposition of $\RR^{4n}$ into
  $\RR^{2n} \times \RR^{2n}$ is equivalent to a standard linear
  example.
\end{proposition}
\begin{proof}
  Recall that ``equivalent'' for $(\Zp)^2$-representations $V$ and $W$ means there is an equivariant isomorphism of vector spaces $V \to W$.

  By virtue of its being abelian, $\Gamma$'s irreducible complex
  representations are one-dimensional.  Let $\chi$ be a character
  of an irreducible complex representation.  Since the map
  $x \mapsto x^2$ is a bijection and since $p \neq 2$, the
  Frobenius-Schur indicator of $\chi$ is
  \[
    \frac{1}{|\Gamma|} \sum_{\gamma \in \Gamma} \chi(\gamma^2) = 
    \frac{1}{|\Gamma|} \sum_{\gamma \in \Gamma} \chi(\gamma) = \begin{cases}
      1 & \mbox{if $\chi$ is trivial,} \\
      0 & \mbox{otherwise.}
    \end{cases}
  \]
  Consequently, a real irreducible representation of $\Gamma$ is
  either the trivial representation or a complex representation
  \cite[page 108]{MR0450380}.  In other words, a nontrivial real
  irreducible representation has character $\chi + \bar{\chi}$ for
  some complex irreducible representation with character $\chi$.  Then
  for $\Gamma = (\Zp)^2$ generated by $(1,0)$ and $(0,1)$, there are
  integers $r$ and $q$, not both zero, so that
  \begin{align*}
    \chi(1,0) &= e^{2\pi r i/p}, \\
    \chi(0,1) &= e^{2\pi q i/p}, \\
  \end{align*}
  and consequently $\chi + \bar{\chi}$ is the character of the real
  two-dimensional representation for which $(1,0)$ acts on $\RR^2$ by
  rotating through the angle $2\pi r/p$ and $(0,1)$ by the angle
  $2\pi q/p$.  Write $W_{r,q}$ for this representation.

  To finish, suppose $\Gamma := \Zp \times \Zp$ and $\Gamma$ acts on
  $\R^{2n} \oplus \R^{2n}$ preserving the decomposition, so the representation decomposes as $V_1 \oplus V_2$
  with $\dim V_1 = \dim V_2 = 2n$. Then by Maschke, the
  representation on $V_i$ decomposes into a sum of irreducible real
  representations.  Each real representation is either trivial or
  $W_{r,q}$ for some $r,q$.  After gathering together any trivial
  representations in pairs as $W_{0,0}$, we conclude that every linear
  representation of $\Gamma$ has the form
  \[
    W_{r_1,q_1} \oplus \cdots \oplus W_{r_n,q_n} \oplus 
    W_{r'_1,q'_1} \oplus \cdots \oplus W_{r'_n,q'_n},
  \]
  i.e., the form given at the beginning of
  Section~\ref{section:constructions}.
\end{proof}

For convenience, define
\[
  W_{R,Q} := W_{r_1,q_1} \oplus \cdots \oplus W_{r_n,q_n} \oplus 
  W_{r'_1,q'_1} \oplus \cdots \oplus W_{r'_n,q'_n},
\]
for the pair of $2n$-tuples $R = (r_1,\ldots,r_n,r'_1,\ldots,r'_n)$
and $Q = (q_1,\ldots,q_n,q'_1,\ldots,q'_n)$.  Also note that
Proposition~\ref{proposition:representation-theory} generalizes to
describe linear representations of $(\Zp)^m$ on $(S^{2n-1})^m$.

As a warning, note again that the action given by these linear
representations need not be a free action.  There is a 
criterion, however, which characterizes when the action is free. 

\begin{lemma}
  Let $B_{ij}$ be the codimensional two plane in the $\Zp$-vector
  space $(\Zp)^{2n}$ consisting of tuples
  $(x_1,\ldots,x_n,x'_1,\ldots,x'_n)$ where $x_i = 0$ and $x'_j = 0$.
  A standard linear example is free exactly when the two-dimensional
  plane $\langle R, Q \rangle \subset (\Zp)^{2n}$ only intersects any
  of the $B_{ij}$ at the origin.
\end{lemma}
\begin{proof}
  Let $W_{R,Q} \cong \RR^{2n} \oplus \RR^{2n}$ be a representation of
  $\Gamma$ as above.  The action on $S^{2n-1} \times S^{2n-1}$
  associated to $W_{R,Q}$ is free exactly when, for all vectors
  $v = (v_1,v_2) \in \RR^{2n} \oplus \RR^{2n}$ with both $v_1 \neq 0$
  and $v_2 \neq 0$, the vector $v$ has trivial stabilizer.  Stated the
  other way, this means that for all $\gamma \in \Gamma$, the
  restriction of $W_{R,Q}$ from $\Gamma$ to $\langle \gamma \rangle$
  decomposes as a sum of representations---but there is a trivial
  representation in the decomposition of at most one term of
  $\RR^{2n} \oplus \RR^{2n}$.

  Let $\gamma = (\gamma_1,\gamma_2) \in (\Zp)^2$.  Then the
  two-dimensional real representation $W_{r,q}$ restricted to
  $\langle \gamma \rangle$ is the representation on $\RR^2$ by which
  $\gamma$ acts by rotation through an angle of
  $2 \pi (\gamma_1 r + \gamma_2 q) / p$.  So the action is free
  exactly when, for all nontrivial $(\gamma_1,\gamma_2)$, there are
  some $i$'s so that
  \[
    \gamma_1 r_i + \gamma_2 q_i \equiv 0 \mod p,
  \]
  or some $j$'s so that 
  \[
    \gamma_1 r'_j + \gamma_2 q'_j \equiv 0 \mod p,
  \]
  but no pair $(i,j)$ which satisfies both conditions.  This is
  exactly the condition given in the lemma.
\end{proof}
In analogy with the lens space case, when the action given by these linear representations is free, we write the quotient as $L(p,p; R, Q)$.

\section{Homotopy classification}
\label{section:homotopy}
  
The homotopy type of a CW complex in which the action of $\pi_1$ on all homotopy groups is trivial is determined by its homotopy groups and a sequence of cohomology classes called $k$-invariants. In the lens space case of free $\Zp$ actions on $S^{2n-1}$, all of the quotients for fixed $p$ and $n$  have the same homotopy groups, hence the homotopy type is determined by the $k$-invariants. It turns out that the first nontrivial $k$-invariant of a lens space is the product of its rotation numbers times a cohomology class in $H^{n+1}(\Zp;\Z)$, and it is the only $k$-invariant necessary to determine the homotopy type \cite{em}. It is also the case (cf.~\cite{firstpaper}) that the first nontrivial $k$-invariant determines the homotopy type of the $L(p,p;R,Q)$ .

In this section, we start with the cohomology ring for $K(\Zp\times\Zp,1)$, then we describe the $k$-invariant to provide the homotopy classification.  A torsion calculation reveals that the homotopy and simple homotopy classifications coincide.

\subsection{The first $k$-invariant and the homotopy type}

The first nontrivial $k$-invariant for a $(2n-1)$-dimensional lens space with rotation numbers $r_1,\dots, r_n$ is 
\[ k(L)=\prod_{i=1}^n (r_i a)\in H^{2n}(\Zp;\Z)
\]
where $a\in H^2(\Zp;\Z)$ is the generator \cite{em}. We think of the $(r_i a)$ as ``rotation classes'' and then the $k$-invariant is just the cup product of the rotation classes. The first nontrivial $k$-invariant in the case of linear $\Zp \times \Zp$ action on a product of equidimensional odd spheres can be defined similarly.

To begin, we provide the integral cohomology of the fundamental group, i.e., the cohomology of $K(\Zp\times\Zp,1)$, as the generators are needed to describe the first nontrivial $k$-invariant. The ring structure is well known and can be found in \cite{firstpaper}, \cite{MR675422}, and \cite{MR2712167}.

\begin{proposition}\label{proposition:zpzp-cohomology}
The integral cohomology ring of $K(\Zp\times\Zp,1)$ is
\[
  H^\star(K(\Zp\times\Zp,1);\Z)\cong \Z[a,b,c]/(pa,pb,pc,c^2)
\]
where $|a|=|b|=2$ and $|c|=3$.
\end{proposition}
More specifically, the isomorphism in Proposition~\ref{proposition:zpzp-cohomology} is chosen so that $a$ and $b$ correspond to the standard basis $(1,0)$ and $(0,1)$ of $\Zp\times\Zp$. There is a copy of $\Zp \times \Zp \subset \Z[a,b] /(pa,pb)$ contained in the cohomology ring, and an automorphism $\varphi : \Zp\times\Zp \to \Zp\times\Zp$ gives rise to an automorphism $H^\star(\varphi)$ which, when restricted to the copy of $\Zp\times\Zp$, can be identified with $\varphi$.

Just as the first $k$-invariant for a lens space is given by the product of the rotation numbers modulo $p$, the first nontrivial $k$-invariant in the case of $L(p,p; R, Q)$ is the product of rotation classes in $H^2(\Zp\times\Zp; \Z)$.

\begin{lemma}[Lemma~5.1 in \cite{firstpaper}]
  Let $L = L(p,p;R,Q)$ and suppose $p > n$.  Then $k(L) \in H^{2n}(\Zp\times\Zp); \ZZ^2)$ is 
  \[
    \left( \prod_{i=1}^n (r_i a + q_i b), \prod_{i=1}^n (r'_i a + q'_i b) \right),
  \]
  where $a$ and $b$ are generators of $H^2(\Zp\times\Zp; \ZZ)$ as in Proposition~\ref{proposition:zpzp-cohomology}.
\end{lemma}

Combining this with Theorem 3.3 in \cite{firstpaper}, we have the following result.

\begin{proposition}[Theorem~3.3 in \cite{firstpaper}]
\label{proposition:homotopy-classification}
Assume the prime $p>3$ satisfies $p > n + 1$.  For free linear actions of $\Zp \times \Zp$ on
  $S^{2n-1} \times S^{2n-1}$ with a specified identification of the quotient's fundamental group with $\Zp \times \Zp$, the quantity 
    \[
    k(L)=\left( \prod_{i=1}^n (r_i a + q_i b), \prod_{i=1}^n (r'_i a + q'_i b) \right) \in H^{2n}(\Zp \times \Zp; \ZZ^2),
  \]
   modulo automorphisms of $\ZZ^2$, determines the homotopy type of the quotient.
 \end{proposition}
Changing the specified identification of the fundamental group with $\Zp \times \Zp$ amounts to applying the same automorphism to $a$ and $b$, as in the discussion following Proposition~\ref{proposition:zpzp-cohomology}.

We compare Proposition~\ref{proposition:homotopy-classification} to the
corresponding statement for $\Zp$ actions on $S^n$.  In the lens
space case, the homotopy classification (with a specified generator of $\pi_1$) boils down to the
product of the rotation classes, i.e., rotation numbers, up to
automorphisms of $\ZZ$, i.e., up to sign. Often, the classification of lens spaces up to homotopy equivalence is described using a formula such as
\begin{equation}\label{eqn:lens-space-classification}
t^n r_1\dots r_n\equiv \pm r'_1\dots r'_n \pmod p,
\end{equation}
where $t$ is an element of $\Zp$ relatively prime to $p$, and $r$ and $r'$ are the rotation numbers corresponding to two lens spaces.
This ``numeric'' formula might be contrasted with the cohomological perspective presented in Proposition~\ref{proposition:homotopy-classification}.  The rotation numbers are now described as rotation classes; the power $t^n$ accounts for a change of identification of $\pi_1$ with $\Zp \times \Zp$,
which would be played by replacing $a$ and $b$ in the formula in 
Proposition~\ref{proposition:homotopy-classification}.  The sign in Equation~\eqref{eqn:lens-space-classification} corresponds to a choice of automorphism of $\Z$, the top dimensional cohomology group, which in Proposition~\ref{proposition:homotopy-classification} becomes an automorphism of $\Z^2$.

\subsection{Simple homotopy classification}
\label{section:simple-homotopy}

In the lens space case, the simple homotopy type is determined by the rotation numbers. Specifically, $L(p;r_1,\dots, r_n)\simeq_s L(p;r_1'\dots r_n')$ if and only if for some $k\in\Zp$ and permutation $\sigma$, $r_i\equiv k r'_{\sigma(i)}$ for all $i$ \cite{whiteheadtorsion}. 

In the case of $L(p,p;R,Q)$, because the quotients are even dimensional, the homotopy classification and the simple homotopy classification coincide.

\begin{proposition}\label{she}
Suppose $f:X\to Y$ is a homotopy equivalence between two of $(S^{2n-1}\times S^{2m-1})/G$'s, where $n\leq m$, and $G$ is a finite abelian group that acts freely on $S^{2n-1} \times 
S^{2m-1}$ and trivially on the cohomology of $S^{2n-1}\times S^{2m-1}$. If $\Wh(G)$ is torsion-free, then $f$ is a simple homotopy equivalence.
\end{proposition}

The proof carries through almost exactly the same as the proof for Proposition~2 in \cite{MR2721633}. We include a version here for completeness.

\begin{proof}
  The spaces $X$ and $Y$ are both simple Poincar\'{e} complexes since they are
  both manifolds, hence they are finite, connected, CW
  complexes with fundamental classes $[X]$ and $[Y]$, respectively,
  and chain homotopy equivalences $\phi_X:C^{2n+2m-2-*}(X)\to C_*(X)$ and
  $\phi_Y:C^{2n+2m-2-*}(Y)\to C_*(Y)$ satisfying $\tau (\phi_X)=0$ and
  $\tau (\phi_Y)=0$.  Here $\tau$, the torsion of the chain equivalence, vanishes by the fact these are \textit{simple} Poincar\'e complexes.

Similar to the work in \cite{davis-loeffler}
with $\pi_1X\cong G$ finite and acting trivially on the cohomology of the universal cover $\tilde{X}$, 
we have the chain-homotopy commutative diagram
\[
\xymatrix{
C^{2n+2m-2-*}(X) \ar[r]^{\hspace{24pt}\phi_X} & C_*(X)\ar[d]^{f_*}\\
C^{2n+2m-2-*}(Y) \ar[u]_{f^*}\ar[r]^{\hspace{24pt}\phi_Y} & C_*(Y)}
\]
and so we have
  \begin{align*}
   0=\tau (\phi_Y)=\tau (f_*\circ\phi_X\circ f^*)&=\tau (f_*)+\tau (\phi_X)+\tau (f^*)\\
   &=\tau (f_*)+\tau (f^*)=\tau (f)+(-1)^{2n+2m-2}\overline{\tau (f)},
  \end{align*}
where $\overline{\tau (f)}$ is the result of applying the involution on $\Wh(\pi_1Y)$ to $\tau(f)$. 
Since $2n+2m-2$ is even and the involution on $\Wh(G)$ is trivial when $G$ is a finite abelian group \cite{bak}, \[
  \tau (f)=-\overline{\tau (f)}=-\tau (f).
\]
As $\Wh(G)$ is torsion-free by assumption, $\tau (f) =0$.
\end{proof}

For $\Gamma=\Zp\times\Zp$ actions on $S^{2n-1}\times S^{2n-1}$, provided $p > 3$ the action on cohomology is trivial, and so we immediately get the following result.
\begin{corollary}\label{corollary:she}
  For $\Gamma=\Zp\times\Zp$ with $p > 3$, two quotient spaces resulting from a free $\Gamma$ action on $S^{2n-1}\times S^{2n-1}$ that are homotopy equivalent are also simple homotopy equivalent.
\end{corollary}
\begin{proof}
Note $\Gamma=\Zp\times\Zp$ is a finite, torsion abelian group that acts trivially on cohomology of $S^{2n-1}\times S^{2n-1}$.  Since $\Wh(\Gamma)=\Wh(\Zp\times\Zp )$ is torsion-free \cite{MR0244257}, the result follows.
\end{proof}

\begin{remark}
  Suppose $X_1$ and $X_2$ are homotopy equivalent but not simple
  homotopy equivalent lens spaces, and $Y_1$ and $Y_2$ are another
  such pair.  Then for $i = 1,2$ the product $X_i \times Y_i$ is the
  quotient of free a $\Zp\times\Zp$ action on $S^{2n-1}\times S^{2n-1}$, and by
  Corollary~\ref{corollary:she}, the product $X_1 \times Y_1$ is
  simple homotopy equivalent to $X_2 \times Y_2$.
\end{remark}

\section{Homeomorphism classification}
\label{section:homeomorphism}

For $\Gamma=\Zp\times\Zp$ and $X$ a quotient space obtained from a free, linear action of $\Gamma$ on $S^{2n-1}\times S^{2n-1}$ as described above, the simple structure set $S^{\Cat,s}(X)$, for $\Cat$ being $Top$ or $PL$, is determined by the $p$-localized Pontrjagin classes. In this section we calculate the Pontrjagin classes explicitly, determine the set of normal invariants, and them determine the full classification.

\subsection{Classification by characteristic classes}

Again $\Gamma$ acts on $S^{2n-1} \times S^{2n-1}$ via a
decomposition-preserving orthogonal representation on $V \oplus V$.
Let $\alpha : \Gamma \to SO(V \oplus V)$ be this representation.  In
the special case $\Gamma = \Zp \times \Zp$, combining knowledge of the
representations of this abelian group with $\dim V = 2n$ refines the
map $\alpha$ to
\[
\Gamma \to SO(2)^{n} \times SO(2)^{n} \subset SO(4n).
\]
This description can be related to the above description: inside each
$SO(2)$ factor is a copy of $\Zp$, so the map $\alpha$ amounts to a
linearization of the above description in terms of $R$ and $Q$,
which are both vectors of $2n$ elements of $\Zp$.

By constructing $BG$ appropriately, regard
$(S^{2n-1} \times S^{2n-1})/\Gamma$ as a subcomplex of $BG$, and
therefore build the map
\[
B\alpha |_{(S^{2n-1} \times S^{2n-1})/\Gamma} : (S^{2n-1} \times S^{2n-1})/\Gamma \to BSO(4n)
\]
The main theorem in \cite{MR0172303} implies that
\[
T(S^{2n-1} \times S^{2n-1}/\Gamma) \oplus \R^2
\]
is a $(4n)$-dimensional vector bundle over
$(S^{2n-1} \times S^{2n-1})/\Gamma$ with classifying map
$B\alpha |_{(S^{2n-1} \times S^{2n-1})/\Gamma}$.

The above description is explicit enough to compute the Pontrjagin
classes.  Suppose $R = (r_1,\ldots,r_n,r'_1,\ldots,r'_n)$ and
$Q = (q_1,\ldots,q_n,q'_1,\ldots,q'_n)$ be elements of $(\Zp)^{2n}$
so that the map $\alpha$ can be defined via $\alpha(1,0) = R$ and
$\alpha(0,1) = Q$ and the inclusion $\Zp \subset SO(2)$.

Consider the effect of $\alpha : B\Gamma \to BSO(4n)$ on cohomology,
namely 
\[
\alpha^\star : H^\star(BSO(4n);\Z) \to H^\star(\Gamma;\Z).
\]
Corresponding to each $SO(2)$ factor in the left hand term of
$SO(2)^n \times SO(2)^n \subset SO(4n)$, there is
$v_i \in H^2(BSO(4n);\Z)$; similarly each $SO(2)$ in the right hand
factor results in $v'_i \in H^2(BSO(4n);\Z)$.  Then
$\alpha^\star(v_i) = r_i a + q_i b$ and
$\alpha^\star(v'_i) = r'_i a + q'_i b$.

\begin{proposition}
\label{proposition:total-class}
The total Pontrjagin class in
$H^\star\left(B\left(SO(2)^n \times SO(2)^n\right);\Z\right)$ is given
by 
\[
\prod_{i=1}^n \left( 1 + {v_i}^2 \right)\left( 1 + {v'_i}^2 \right)\]
implying that
\[
  p( \left(S^{2n-1} \times S^{2n-1}\right)/\Gamma) = 
  \prod_{i=1}^n \left( 1 + \left(r_i  \bar{a} + q_i  \bar{b}\right)^2 \right)
        \left( 1 + \left(r'_i \bar{a} + q'_i \bar{b}\right)^2 \right).
\]
\end{proposition}
There is the map from the quotient $\left(S^{2n-1} \times S^{2n-1}\right)/\Gamma$ to the classifying space $B\Gamma$ which, in cohomology, provides a map from $H^\star(B\Gamma)$ to the cohomology of the quotient.  We have denoted $a$ and $b$ as the generators of $H^2(B\Gamma;\Z)$, and in what follows, let $\bar{a}$ and $\bar{b}$ denote the image of the generators under the map
\[
H^\star(B\Gamma) \to H^\star(\left(S^{2n-1} \times S^{2n-1}\right)/\Gamma).
\]
Compare the computation in Proposition~\ref{proposition:total-class} with the computation for the total Pontrjagin class of the lens space $L$, whcih is given by
\[
  p( L) = 
  \prod_{i=1}^n \left( 1 + (r_i \bar{a})^2 \right)
\]
where the $r_i$ are the rotation numbers.

\subsection{Surgery}

To understand the possible actions of $\Gamma=\Zp \times \Zp$ on $S^{2n-1} \times S^{2n-1}$, we consider the surgery exact sequence
\[
\cdots \to L^s_{4n-1}(\Gamma) \to S^{\Cat,s}(X) \to [X, \GModCat] \to L_{4n-2}^s(\Gamma).
\]
where $X=(S^{2n-1}\times S^{2n-1})/\Gamma$ and $\GModCat$ is $\GModTop$ or $\GModPL$.
We note that this is the simple surgery exact sequence and that the $L_m^s$ are simple L-groups.

In general, for a group $\Gamma$ of odd order, the simple L-groups are 
$L_m^s(\Z \Gamma)=\Sigma\oplus 8\Z, 0, \Sigma\oplus \Z/2,0$ for $m \equiv 0,1,2,3 \pmod 4$, where $\Sigma$ has no torsion \cite{bak}. Since $4n-2\equiv 2 \pmod 4$, our surgery exact sequence becomes:

\[
0 \to S^{\Cat,s}(X) \to [X, \GModCat] \to \Sigma\oplus \Z/2,
\]
and our classification comes down to determining the set of normal invariants, $[X,\GModCat]$, and the map $[X, \GModCat] \to L_{4n-2}^s(\Gamma)=\Sigma\oplus\Z/2$, for $X=L(p,p;Q,R)$. We start first with the set of normal invariants.

\subsection{Normal invariants}

We first need a couple of lemmas. We note that they are similar to lemmas in \cite{MR2721633}. 

\begin{lemma}
  \label{lemma:ahss}
  Suppose that $\Gamma=\Zp\times\Zp$ and that $Y$ is a
  $\Cat$-manifold for $\Cat = \PL$ or $\Top$ and with a free $\Gamma$-action such that $\Gamma$ acts
  trivially on $H^\star(Y ; \Z[1/p])$.  Then
  \[
    [Y/\Gamma, \GModCat] \left[ 1/p \right] \cong
    [Y, \GModCat] \left[1/p \right].
  \]
\end{lemma}

\begin{proof}
    Let $\pi :Y\to Y/\Gamma$ be the projection and let $\pi_*:H_*(Y;\Z [\frac{1}{p}])\to H_*(Y/\Gamma;\Z [\frac{1}{p}])$ be the induced map on homology. The transfer gives a map going the other way: $tr:H_*(Y/\Gamma;\Z [\frac{1}{p}])\to H_*(Y;\Z [\frac{1}{p}])$.  
    Then $tr \circ \pi_* : H_*(Y;\Z[\frac{1}{p}]) \to H_*(Y;\Z[\frac{1}{p}])$ is multiplication by $|\Gamma|$ because $\Gamma$ acts trivially on $H_\star(Y;\Z[\frac{1}{p}])$.  Consequently, $\frac{1}{|\Gamma|} tr \circ \pi_*=id$ and we also have that $\pi_*\circ\frac{tr}{|\Gamma|}=id$. Since $Y$ is finite dimensional we see that there is an isomorphism between $H_\star(Y;\Z[\frac{1}{p}])$ and $H_\star(Y/\Gamma;\Z[\frac{1}{p}])$, and by Poincar\'e duality, we obtain an isomorphism between $H^*(Y;\Z[\frac{1}{p}])$ and $H^\star(Y/\Gamma;\Z[\frac{1}{p}])$.

    As $\GModCat$ is a spectrum, maps into the spaces of $\GModCat$ are naturally equivalent to a generalized cohomology theory, i.e. $E^n(X)\cong$ $[X,\GModCat_n]$.  It follows from the Atiyah-Hirzebruch spectral sequence that $H^*(Y;E_*(*;\Z [\frac{1}{p}])) \cong H^*(Y/\Gamma; E_*(*;\Z [\frac{1}{p}]))$ and $[Y,\GModCat[\frac{1}{p}]]\cong[Y/\Gamma,\GModCat[\frac{1}{p}]]$.

\end{proof}

\begin{corollary}\label{corollary:zmod2}
Suppose $\Gamma=\Zp\times\Zp$ acts freely on $S^{2n-1} \times S^{2n-1}$.  Then for $\Cat$ being $\Top$ or $\PL$, 
\[
\left[(S^{2n-1} \times S^{2n-1})/\Gamma,\GModCat\right] \left[ 1/p \right] \cong \Zmod{2}.
\]
\end{corollary}
\begin{proof}
  By Lemma~\ref{lemma:ahss}, 
  \begin{align*}
    \left[(S^{2n-1}  \times S^{2n-1})/\Gamma,\GModCat \right] \left[ 1/p \right] 
    &\cong \left[S^{2n-1} \times S^{2n-1},\GModCat\right] \left[ 1/p \right] \\
    \cong \pi_{2n-1}\left( \GModCat\left[ 1/p \right] \right) &\oplus
      \pi_{2n-1}\left( \GModCat\left[ 1/p \right] \right) \oplus
      \pi_{4n-2}\left( \GModCat\left[ 1/p \right] \right).
  \end{align*}
  To identify maps out of the product $S^{2n-1} \times S^{2n-1}$ with the given sum of homotopy groups, we need to verify that the Whitehead products vanish; this follows from the fact that $\GModCat$ is an path-connected H-space and that localizations of such spaces are path-connected H-spaces.  The odd-dimensional homotopy groups of $\GModCat$ are trivial, so the localized group $\pi_{2n-1}\left( \GModCat\left[ 1/p \right] \right)$ also vanishes.  Since $4n-2 \equiv 2 \pmod 4$ and since $\pi_{4n-2}\left(\GModCat\right) = \Zmod{2}$, we have $\pi_{4n-2}\left( \GModCat\left[ 1/p \right] \right) = \Zmod{2}$ and the conclusion follows.
\end{proof}

  In \cite{MR2721633}, Lemma 6 gives that for odd
  primes $p > 2n$, there is a $(2n+2)$-equivalence
  \[
    (BO)_{(p)} \to \prod_{0 < j < n/2} K(\Z_{(p)}, 4j).
  \]
  Using the nice description of $BO$ from Theorem~7.4 of the survey article \cite{MR1747531},
Lemma 6 from \cite{MR2721633} can be improved:

\begin{lemma}
\label{lemma:2p1-equiv}

  For an odd prime $p$, there is a $(2p+1)$-equivalence
  \begin{align*}
    (BO)_{(p)} &\to K( \mathbb{Z}_{(p)}, 4 ) \times K( \mathbb{Z}_{(p)}, 8 ) \times \cdots \times K( \mathbb{Z}_{(p)}, 2(p-1) ) \\
    &\simeq \prod_{i=1}^{(p-1)/2} K( \mathbb{Z}_{(p)}, 4i ).
  \end{align*}
\end{lemma}

\begin{proof}
    For an odd prime $p$, there is an equivalence of $H$-spaces
  \[
    (BO)_{(p)} \simeq \prod_{i=0}^{(p-3)/2} \Omega^{4i} W,
  \]
  where $W$ has trivial homotopy except for
  $\pi_{2i(p-1)} (W) = \mathbb{Z}_{(p)}$ for positive integers $i$ \cite{MR1747531}.    For $p$ an odd prime, $4(p-1) > 2p+1$, and therefore
  \[
    W \to K( \mathbb{Z}_{(p)}, 2(p-1) )
  \]
  is a $(2p+1)$-equivalence.  Moreover, for any $0 \leq i \leq (p-3)/2$,
  then $4(p-1) - 4i \geq 2p + 2> 2p+1$, and so
  \[
    \Omega^{4i} W \to K( \mathbb{Z}_{(p)}, 2(p-1) - 4i )
  \]
  is a $(2p+1)$-equivalence.  Therefore
  \begin{align*}
    (BO)_{(p)}
    &\to \prod_{i=0}^{(p-3)/2} \Omega^{4i} K( \mathbb{Z}_{(p)}, 2(p-1) ) \\
    &\simeq \prod_{i=0}^{(p-3)/2} K( \mathbb{Z}_{(p)}, 2(p-1) - 4i ) \\
    &\simeq \prod_{i=1}^{(p-1)/2} K( \mathbb{Z}_{(p)}, 4i )
  \end{align*}
  is a $(2p+1)$-equivalence.
\end{proof}

For example, when $p = 3$, we have that
\[
(BO)_{(3)} \simeq K(\Z_{(3)}, 4)
\]
is a 7-equivalence.

\begin{lemma}\label{lemma:normalinvts}
Suppose that $\Gamma=\Zp\times\Zp$ acts freely on $S^{2n-1} \times S^{2n-1}$ with quotient $X$. Then 
\[
\left[X,\GModCat\right] \cong \Zmod{2} \oplus \bigoplus_i H^{4i}(X;\Zp)
\]
for $\Cat =$ $\PL$ or $\Top$.
\end{lemma}
\begin{proof}
Localizing $\GModCat$ at $p$ and away from $p$ gives rise to a localization square
\[
\begin{tikzcd}
\left[X,\GModCat\right] \arrow[r] \arrow[d] & \left[X,\GModCat\right] \left[ 1/p \right] \arrow[d] \\
\left[X,\GModCat \right]_{(p)} \arrow[r] & \left[X,\GModCat \right]_{(0)}
\end{tikzcd}
\]
and hence then following exact sequence
\[
0
\to \left[X,\GModCat\right]
\to \left[X,\GModCat\right] \left[ 1/p \right] \oplus \left[X,\GModCat\right]_{(p)}
\to \left[X,\GModCat\right]_{(0)}
\to 0.
\]
  Since $\left[X,\GModCat\right]_{(0)} \cong \bigoplus_i H^{4i}(X;\Q) \cong 0$, 
  
  the short exact sequence from the localization square yields
  \[
    \left[X,\GModCat\right] \cong \left[X,\GModCat\right]\left[ 1/p \right] \oplus \left[X,\GModCat\right]_{(p)}.
  \]
  By Corollary~\ref{corollary:zmod2},
  \[
    \left[X,\GModCat\right]\left[ 1/p \right] \cong \Zmod{2},
  \]
  a Kervaire-Arf invariant.   But $BO_{(p)} \cong \GModCat_{(p)}$, so both localizations are a product of Eilenberg-MacLane spaces and the lemma then follows from the calculation of $\left[X,\GModCat\right]_{(p)}$.
\end{proof}

\subsection{The Structure Set}

Recall that the surgery exact sequence for $X=L(p,p;Q,R)$ becomes
\[
0 \to \mathcal{S}^{\Cat,s}(X) \to [X, \GModCat] \to \Sigma \oplus\Zmod{2}.
\]

From Lemma \ref{lemma:normalinvts}, we have that the right most map in the surgery exact sequence above is $[X, \GModCat]\cong\Zmod{2} \oplus H^{4i}(X;\Zp) \to \Sigma \oplus \Zmod{2}$. Both $\GModTop$ and $\GModPL$ have H-space structures that make this map into a homomorphism. Since $\Sigma$ is torsion-free, and $H^{4i}(X;\Zp)$ is trivial or $p$-torsion for all $i$, all of the $p$-torsion in $H^{4i}(X;\Zp)$ must come from the structure set.

Consider the map $X\to S^{4n-2}$ given by sending everything but a disk to a point. This induces a commutative diagram
\[
\xymatrix{
[S^{4n-2},\GModCat]\ar@{^{(}->}[r]\ar[d] & [X,\GModCat]\ar[d]\ar@{->>}[r] & [(S^{2n-1})^2,\GModCat]\ar[d]\\
L^s_{4n-2}(e)\ar@{^{(}->}[r] & L^s_{4n-2}(\Zp\times \Zp)\ar@{->>}[r] & L^s_{4n-2}(e).
}
\]
By the Kerviare-Arf invariant, $\pi_{4n-2}(\GModCat)$ maps isomorphically into the $L$-group of the trivial group, and we see that the $\Zmod{2}$ in the set of normal invariants maps isomorphically onto the $\Zmod{2}$ in $L_{4n-2}(\Gamma)$.

From this discussion, we conclude that $S^{\Cat,s}(X)=\bigoplus_i H^{4i}(X;\Zp)$, and the classification is determined by the Pontrjagin classes.  This yields Theorem~\ref{theorem:homeomorphism}

\begin{theorem}\label{theorem:homeomorphism}
Suppose $p>3$ and $p>n+1$, and following Proposition~\ref{proposition:homotopy-classification}, suppose $X$ and $Y$ are quotients of linear actions of $\Zp \times \Zp$ on $S^{2n-1} \times S^{2n-1}$ and the map $f : X \to Y$ is a homotopy equivalence.   Then $f$ is homotopic to a homeomorphism provided the Pontrjagin class $p(Y) \in H^\star(Y;\Zmod{p})$ pulls back via $f^\star$ to $p(X) \in H^\star(X;\Zmod{p})$.
\end{theorem}

Since $f$ is a homotopy equivalence, the rings $H^\star(X;\Zmod{p})$ and $H^\star(Y;\Zmod{p})$ are isomorphic, and can be computed using knowledge of the $k$-invariant.  Specifically, one can compute the cohomology of $L(p,p;R,Q)$ using the Serre spectral sequence on the Borel fibration $S^{2n-1} \times S^{2n-1} \to L(p,p;R,Q) \to B\Gamma$ and the relationship between the $k$-invariant and the transgression.

\section{Applications and Open Questions}
\label{section:application}

Kwasik--Schultz \cite{MR1876893} classified squares of 3-dimensional lens spaces up to diffeomorphism: for a prime $p > 3$ and rotation numbers $r$ and $q$, there is a diffeomorphism
\[
  L(p;1,r) \times L(p;1,r) \cong L(p;1,q) \times L(p;1,q).
\]
The machinery of this paper can be used to provide a homeomorphism in that case, and also extend the classification to 
products of \textit{different} lens spaces.
\begin{theorem}\label{thm:application}
  For a prime $p > 3$ and rotation numbers $r_1$, $r_2$, $q_1$, $q_2$ such that $\pm (r_1 r_2) / (q_1 q_2)$ is a quadratic residue mod $p$, there
  is a homeomorphism
  $L(p;1,r_1) \times L(p;1,r_2) \cong L(p;1,q_1) \times L(p;1,q_2)$.

  Conversely, if there is such a homeomorphism, then the above condition on rotation numbers is satisfied.
\end{theorem}
\begin{proof}
    Let's first show that they are homotopy equivalent; this was done with somewhat more machinery in \cite{firstpaper} but for completeness, here we present a proof relying only on Proposition~\ref{proposition:homotopy-classification}.  In the notation of Section~\ref{section:constructions},
  \begin{align*}
    L(p;1,r_1) \times L(p;1,r_2) &= L(p,p;R_1,R_2), \\    
    L(p;1,q_1) \times L(p;1,q_2) &= L(p,p;Q_1,Q_2),
  \end{align*}
  where
  \begin{align*}
    R_1 &= (1,r_1,0,0), \\
    R_2 &= (0,0,1,r_2), \\
    Q_1 &= (1,q_1,0,0), \\
    Q_2 &= (0,0,1,q_2). \\
  \end{align*}
  Invoke Proposition~\ref{proposition:homotopy-classification} to
  conclude \(L(p,p;R_1,R_2) \simeq L(p,p;Q_1,Q_2)\).
  The argument proceeds via $k$-invariants.  Specifically,
  \begin{align*}
    k(L(p,p;R_1,R_2)) &= \left( r_1 a^2, r_2 b^2 \right) \in H^4( (\Zp)^2 ; \ZZ^2 ) \\
    k(L(p,p;Q_1,Q_2)) &= \left( q_1 a^2, q_2 b^2 \right) \in H^4( (\Zp)^2 ; \ZZ^2 ) \\
  \end{align*} 
  By the hypothesis, choose $z$ so that
  $r_1 r_2 z^2 = \pm q_1 q_2$.   Set
  $\lambda := q_1 / (r_1 z^2)$ and set
  $\mu := q_2 / r_2$.  Then $\lambda \mu \equiv \pm 1 \mod p$ and since the map 
  $\GL_2(\Z) \to \{ M \in \GL_2(\Z/p) : \det M = \pm 1 \}$ is surjective, there is an automorphism of $\Z^2$ which, when reduced modulo $p$, amounts to multiplication by $\lambda$ and $\mu$ on the two generators, respectively.  Note that $\GL_2(\Z)$ consists of matrices with integer entries and determinant $\pm 1$, and the map $\GL_2(\Z) \to \{ M \in \GL_2(\Z/p) : \det M = \pm 1 \}$ is reduction modulo $p$.
  
  So the $k$-invariant $k(L(p,p;R_1,R_2)) = \left( r_1 a^2, r_2 b^2 \right)$ and the $k$-invariant $\left( \lambda r_1 a^2, \mu r_2 b^2 \right)$ correspond to homotopy equivalent spaces, but the latter $k$-invariant is  $\left( q_1 a^2 / z^2, q_2 b^2 \right)$, which, after changing the chosen generator of $\pi_1$, yields the $k$-invariant $k(L(p,p;Q_1,Q_2)) = \left( q_1 a^2, q_2 b^2 \right)$. 

To finish the homeomorphism classification, we must compare Pontrjagin classes modulo $p$. 
  The total Pontrjagin class of $L(p,p;R_1,R_2)$ is
  \[
  (1+\bar{a}^2)(1+{r_1}^2 \bar{a}^2) (1+\bar{b}^2)(1+{r_2}^2 \bar{b}^2),
  \]
  which is concentrated in $H^4$ since $L(p,p;R_1,R_2)$ 
 is 6-dimensional, so the Pontrjagin class is
  \[
  \bar{a}^2 + {r_1}^2 \bar{a}^2 + \bar{b}^2 + {r_2}^2 \bar{b}^2 \in H^4(L),
  \]
  but $\bar{a}^2$ and $\bar{b}^2$ both vanish, so this Pontrjagin class is zero.  There are other ways to see that this vanishes: twice the Pontrjagin class of $L(p,p;R_1,R_2)$ is twice the product of the Pontrjagin classes of the two lens spaces, but the Pontrjagin class of a 3-dimensional lens spaces is trivial.

  For the converse direction, note that a homeomorphism necessarily results in a homotopy equivalence, and with the discussion of $k$-invariant in Section~6 of \cite{firstpaper}, the condition on the rotation numbers must be satisfied. 
\end{proof}

Theorem~\ref{thm:application} yields a homeomorphism, so a natural question is whether one can produce a diffeomorphism.  One could combine the proof of Theorem~1.1 of \cite{MR1876893} with the simple homotopy equivalence provided by Propositions~\ref{proposition:homotopy-classification} and~\ref{she} to produce a diffeomorphism.

However, several questions remain unanswered and present opportunities for future work.  For instance, theorem~\ref{thm:application} examined the case of $S^3 \times S^3$; a classification in the case of $S^5 \times S^5$ remains open.  In that case, instead of studying pairs of quadratic forms, the homotopy classification involves cubic forms.  And in addition to studying diffeomorphism, one could seek a classification up to \textit{almost} diffeomorphism, meaning up to connected sums with exotic spheres.  It would also be interesting to extend Proposition~\ref{proposition:total-class} to include not only information on the Pontrjagin class, but also information on the action of the group of homotopy self-equivalences on the Pontrjagin class.

\printbibliography

\end{document}